\documentclass[a4paper]{amsart}
\pdfoutput=1

\usepackage{amsthm}

\newtheorem{theorem}{Theorem}

\newtheorem{lemma}[theorem]{Lemma}

\newtheorem{question}[theorem]{Question}
\newtheorem{proposition}[theorem]{Proposition}

\usepackage{hyperref}

\title{Functions of constant geodesic X-ray transform}
\author{Joonas Ilmavirta}
\thanks{Department of Mathematics and Statistics, University of Jyv\"askyl\"a, Finland. \texttt{joonas.ilmavirta@jyu.fi}}
\author{Gabriel P. Paternain}
\thanks{Department of Pure Mathematics and Mathematical Statistics, University of Cambridge, UK. \texttt{g.p.paternain@dpmms.cam.ac.uk}}
\date{\today}

\newcommand{\Z}{\mathbb Z}

\newcommand{\R}{\mathbb R}

\newcommand{\eps}{\varepsilon}
\newcommand{\abs}[1]{\left\lvert #1 \right\rvert}
\newcommand{\der}{\mathrm d}
\newcommand{\Der}[1]{\frac{\der}{\der #1}}
\newcommand{\A}{\mathcal A}
\newcommand{\sff}{I\!I}
\newcommand{\ip}[2]{\left\langle#1,#2\right\rangle}
\newcommand{\loc}{{\text{loc}}}
\newcommand{\Euc}{{\text{Eucl}}}

\newcommand{\bd}[1]{{\check{#1}}}
\newcommand{\dbd}[1]{\dot{\bd{#1}}}

\newcommand{\Order}{\mathcal O}

\begin{document}

\begin{abstract}
We show that the existence of a function in $L^{1}$ with constant geodesic X-ray transform imposes geometrical restrictions on the manifold. The boundary of the manifold has to be umbilical and in the case of a strictly convex Euclidean domain, it must be a ball. Functions with constant geodesic X-ray transform always exist on manifolds with rotational symmetry.
\end{abstract}



\maketitle

\section{Introduction}

We study functions whose integral is~$1$ over every maximal geodesic on a simple Riemannian manifold.
A compact manifold is called simple if it is simply connected, there are no conjugate points, and the boundary is strictly convex.
It turns out that the existence of such a function severely restricts the geometry of the manifold.
The boundary has to be umbilical and a Euclidean domain with such a function is necessarily a ball.
Such functions do exist on Euclidean balls and spherically symmetric manifolds satisfying the Herglotz condition.
It may well be that this short list of examples is exhaustive.

Existing range characterizations (such as~\cite{PU:range}) concern functions that are smooth up to the boundary.
In our setting the integral over every geodesic is a non-zero constant and geodesics get shorter near the boundary, and therefore the function in question must blow up at the boundary.
Indeed, it follows from \cite[Theorem 2.2]{MNP:bayes-xrt} and injectivity of the geodesic X-ray transform~$I$ on simple manifolds (see e.g. the review~\cite{IM:review}) that if~$M$ is simple and $f\in L^1(M)$ is such that $I^*If\in C^\infty(M)$, then
\begin{equation}
\label{eq:f=dw}
f(x)=d(x,\partial M)^{-1/2}w(x)
\end{equation}
in a neighborhood of the boundary, where $w\in C^\infty(M)$.
Therefore such a function~$f$ is in~$L^1$ but not in~$L^p$ for any $p\geq 2$.
Throughout the paper~$C^\infty(M)$ refers to the space of functions that are smooth up to the boundary of the manifold~$M$.
When $f\in L^1(M)$, the notation $If\equiv1$ means that the integral is one over almost all geodesics, and similarly for~$I^*I$.
(The adjoint~$I^*$ is defined using a natural $L^2$ inner product at the boundary.) 

We have three main results.
We state them below and prove them in the subsequent sections.
The first one is a boundary determination result for functions of constant X-ray transform.
The boundary determination concerns both the values of the function near the boundary (in a rescaled sense due to the blowup) and the shape of the boundary itself.

We recall that a boundary point is called umbilical if the second fundamental form is conformal to the first fundamental form.
That is, the second fundamental form only depends on the magnitude, not direction, of a tangent vector of the boundary.
When we say that the boundary is umbilical, we mean that every boundary point is an umbilical one but the conformal factor need not be constant.

\begin{theorem}[Proven in section~\ref{sec:bdy-det}]
\label{thm:bdy-det}
Let~$M$ be a simple Riemannian manifold of dimension $n\geq2$.
Suppose $f\in L^1(M)$ integrates to~$1$ over almost all maximal geodesics.
Then the boundary is umbilical.

Moreover,~$f$ is of the form $f(x)=d(x,\partial M)^{-1/2}w(x)$ near the boundary with $w\in C^\infty(M)$ satisfying
\begin{equation}
w(x)
=
\sqrt{\frac{\sff(x)}{2\pi^2}}
\end{equation}
for all $x\in\partial M$, where the second fundamental form is regarded as a scalar function.
\end{theorem}

If the function~$f$ is assumed to be of the form~\eqref{eq:f=dw} for a continuous~$w$, then the conclusion holds locally at any strictly convex boundary point, using only short geodesic near that point.
In the proof of theorem~\ref{thm:bdy-det} integrals over long geodesics are only used to show that $f$ has to have the specific form near the boundary.

If $\dim(M)=2$, then the boundary is always umbilical.
Therefore theorem~\ref{thm:bdy-det} does not constrain the geometry of a simple surface, but it still gives a boundary determination result for the function~$f$.

Our second result is a characterization of all Euclidean domains having a function of constant X-ray transform.

\begin{theorem}[Proven in section~\ref{sec:Rn}]
\label{thm:Rn}
If $\Omega\subset\R^n$, $n\geq2$, is a smooth, bounded, and strictly convex domain, then the following are equivalent:
\begin{enumerate}
\item
There exists a function $f\in L^1(\Omega)$ so that integral of~$f$ over almost every line meeting~$\Omega$ is~$1$.
\item
The domain~$\Omega$ is a ball.
\end{enumerate}
Moreover, the function~$f$ is unique when it exists.
\end{theorem}

In fact, uniqueness holds on any simple manifold, as follows from the mapping properties of the normal operator.
If~$If$ is constant, so is~$I^*If$, and by bijectivity of~$I^*I$ exactly one function maps into a given constant function.

A stronger version of theorem~\ref{thm:Rn} is given in theorem~\ref{thm:Rn-local}.
The conclusion is the same but $If(\gamma)=1$ is only required for short~$\gamma$.

The existence of such functions in the ball (see equation~\eqref{eq:radial-f}) leads to a non-uniqueness phenomenon in all Euclidean domains.
By extending the function from a ball by zero we obtain a function whose X-ray transform obtains two values, and linear combinations of such functions produce functions whose X-ray transform is piecewise constant.
In quantum mechanical contexts, where the X-ray transform~$I$ appears in expressions of the form $\exp(-iIf)$ (see e.g.~\cite{I:qm}), functions whose X-ray transform takes values in~$2\pi\Z$ are transparent.

We expect that theorem~\ref{thm:Rn} holds for the Radon transform as well, with a similar proof using integral moments.
With a slicing argument as the one used for theorem~\ref{thm:Rn-local}, one should be able to generalize the result further to any $d$-plane Radon transform.
It was indeed recently verified~\cite{DS:radon} that if the Radon transform depends only on the distance of the hyperplane to the nearest parallel hyperplane tangent to the boundary, then the domain has to be a ball.
No integrable function has constant Radon transform in~$\R^n$ for $n\geq3$.

Theorem~\ref{thm:Rn} is somewhat similar to the Pompeiu problem. This is a conjecture stating that if a non-zero continuous function integrates to zero over every congruent copy of a simply connected Lipschitz domain, then the domain must be a ball.
For a more detailed description of the problem and a verification of the conjecture in three dimensions, see~\cite{R:pompeiu}.

The third and last result gives examples of non-Euclidean manifolds where functions of constant X-ray transform exist.
The example assumes symmetry, but the Herglotz condition is weaker than simplicity (see e.g.~\cite{M:numeric}).

\begin{proposition}[Proven in section~\ref{sec:radial}]
\label{prop:radial}
If the smooth function $c\colon [0,1]\to(0,\infty)$ satisfies the Herglotz condition
\begin{equation}
\label{eq:Herglotz}
\Der{r}\left(\frac{r}{c(r)}\right)>0,
\end{equation}
then the radial function
\begin{equation}
\label{eq:f(r)}
f(r)
=
\frac{c(r)-rc'(r)}{\pi c(r)\sqrt{c(1)^{-2}-(r/c(r))^2}}
\end{equation}
integrates to~$1$ over every maximal geodesic in the ball $\bar B(0,1)\subset\R^n$, $n\geq2$, with the conformally Euclidean metric $c^{-2}(r)g_\Euc$.
\end{proposition}

We point out that all rotationally symmetric Riemannian metrics on a Euclidean ball are of this form~\cite[Proposition C.1]{dHIK:rigid}.

Due to the Euclidean result of theorem~\ref{thm:Rn}, it seems natural to speculate that the radial nature of the examples given in proposition~\ref{prop:radial} is not accidental, hence we pose the following:


\begin{question}
Let~$M$ be a simple Riemannian manifold with $\dim(M)\geq2$.
If there exists a function $f\in L^1(M)$ that integrates to~$1$ over every geodesic, is it true that~$M$ is spherically symmetric?
\end{question}

\section*{Acknowledgements}
J.I.\ was supported by the Academy of Finland (decision 295853).
G.P.P.\ was supported by an EPSRC grant (EP/R001898/1).
We thank Gunther Uhlmann, Plamen Stefanov, and Mikko Salo for discussions and the anonymous referees for a number of useful suggestions.

\section{Boundary determination}
\label{sec:bdy-det}

\subsection{Boundary normal coordinates}

We use boundary normal coordinates $x=(x^0,\bd x)$, where~$x^0$ is the distance from~$x$ to~$\partial M$ and~$\bd x$ is the closest boundary point to~$x$.
Latin indices $i,j,k,\dots$ go from~$1$ to $n-1$ and they correspond to a local system of coordinates on~$\partial M$.
These coordinates are defined in a neighborhood of the boundary, and we shall be working in such a neighborhood.

The second fundamental form as a function of two vectors $a,b\in T_xM$ can be written in these coordinates as
\begin{equation}
\sff(a,b)
=
-\frac12\frac{\partial}{\partial x^0}g_{ij}(x)a^ib^j.
\end{equation}
This formula extends the second fundamental form from the boundary to a neighborhood: the vectors need not be tangential to the boundary and the base point need not be on the boundary.
The same extension has been used in~\cite{KKL:spectral,I:bdy-det}.

The zeroth component of the geodesic equation is
\begin{equation}
\ddot\gamma^0(t)
+
\sff(\dot\gamma,\dot\gamma)
=
0.
\end{equation}
We will be following the notation and ideas of~\cite{I:bdy-det} quite closely in this section, but we will not use the results obtained in that paper.

\subsection{Asymptotic integrals}

The goal of this subsection is to estimate integrals of functions of the form~\eqref{eq:f=dw} using boundary normal coordinates.

Consider a point $x\in\partial M$ where the second fundamental form is positive definite.
We will work in a small neighborhood of~$x$, and we may shrink the neighborhood implicitly when needed.
All error estimates are uniform near~$x$.

We study short geodesics defined as follows.
Take any $v\in S_x(\partial M)$ (a unit tangent vector) and a small $h>0$.
Parallel translate~$v$ by distance~$h$ along the geodesic normal to the boundary.
(In Euclidean geometry this amounts to shifting the vector~$v$ from~$x$ to $x+h\nu$, where~$\nu$ is the inward pointing unit normal vector.)
Then $\gamma_{v,h}\colon I_{v,h}\to M$ is the maximal geodesic in the obtained direction.
We fix parametrization so that $\gamma_{v,h}^0(0)=h$ and $I_{v,h}=[-\tau_-^{v,h},\tau_+^{v,h}]$.
To avoid clutter, we will mostly denote simply $\gamma_{v,h}=\gamma$.

\begin{lemma}
\label{lma:d-int}
The geodesic defined above satisfies
\begin{equation}
\int_{I_{v,h}}\gamma^0(t)^{-1/2}\der t
=
\sqrt{\frac{2\pi^2}{\sff(v,v)}}
+
\Order(\sqrt{h})
\end{equation}
as $h\to0$ for any fixed~$v$.
\end{lemma}

\begin{proof}
We streamline the notation by leaving out the parameters $v$ and $h$, so that $\gamma_{v,h}=\gamma$ and $\tau_\pm^{v,h}=\tau_\pm$.
We denote $a=\sff(\dot\gamma(0),\dot\gamma(0))$.

It follows from the geodesic equation and Lipschitz continuity of~$\sff$ and Taylor approximations that
\begin{equation}
\label{eq:vv3}
\dot\gamma^0(t)
=
-at+\Order(t^2)
\end{equation}
and
\begin{equation}
\gamma^0(t)
=
h-\frac12at^2+\Order(t^3).
\end{equation}
By strict convexity $a>0$ and so $h-\gamma^0(t)$ is comparable with~$t^2$.
From
\begin{equation}
\gamma^0
=
h-\frac12at^2+\Order((h-\gamma^0)^{3/2})
\end{equation}
we find
\begin{equation}
\label{eq:t}
t
=
\pm\sqrt{\frac2a(h-\gamma^0)+\Order((h-\gamma^0)^{3/2})}
=
\pm\sqrt{\frac2a(h-\gamma^0)}
+
\Order(h-\gamma^0).
\end{equation}
Using this in equation~\eqref{eq:vv3} gives
\begin{equation}
\dot\gamma^0
=
\mp\sqrt{2a(h-\gamma^0)}+\Order(h-\gamma^0)
\end{equation}
and so using the Taylor approximation $(x+y)^{-1}=x^{-1}-x^{-2}y+\Order(x^{-3}y^2)$ leads to
\begin{equation}
\dot\gamma^0(t)^{-1}
=
\mp[2a(h-\gamma^0)]^{-1/2}+\Order(1).
\end{equation}
The different signs correspond to the different halves of the geodesic going down from the tip from the point of view of boundary normal coordinates.

We are now ready to start computing the integral.
We split it in two halves:
\begin{equation}
\int_{I_{v,h}}\gamma^0(t)^{-1/2}\der t
=
\sum_{\pm}\int_0^{\tau_\pm}\gamma^0(\pm t)^{-1/2}\der t.
\end{equation}
On each half we change the integration variable from $t\in[0,\tau_\pm]$ to $z=\gamma^0(t)\in[0,h]$.
The Jacobian is~$\dot\gamma^0(t)^{-1}$ which we computed above.
We get
\begin{equation}
\begin{split}
\int_0^{\tau_\pm}\gamma^0(\pm t)^{-1/2}\der t
&=
\int_0^h z^{-1/2} ([2a(h-z)]^{-1/2}+\Order(1))\der z
\\&=
\frac{\pi}{\sqrt{2a}}
+\Order(\sqrt{h}),
\end{split}
\end{equation}
where we used the substitution $z=h\sin^2\theta$.
This leads to
\begin{equation}
\int_{I_{v,h}}\gamma^0(t)^{-1/2}\der t
=
\frac{2\pi}{\sqrt{2a}}
+
\Order(\sqrt{h}).
\end{equation}
The proof of the lemma is concluded by noticing that $a=\sff(v,v)+\Order(h)$.
\end{proof}

Applying equation~\eqref{eq:t} when $\gamma^0=0$ (where the geodesic meets the boundary) shows that~$\tau_\pm$ is comparable with~$\sqrt{h}$ when~$h$ is sufficiently small.
Therefore $\abs{I_{v,h}}=\Order(\sqrt{h})$.
As $\abs{\dot\gamma(t)}\equiv1$, we have that $\dbd\gamma(t)$ is uniformly bounded and thus $d(\bd \gamma(t),\bd x)=\Order(\abs{I_{v,h}})=\Order(\sqrt{h})$.

\begin{lemma}
\label{lma:dw-int}
For the geodesic defined above and $w\in C^\infty(M)$ we have
\begin{equation}
\int_{I_{v,h}}\gamma^0(t)^{-1/2}w(\gamma(t))\der t
=
\sqrt{\frac{2\pi^2}{\sff(v,v)}}w(x)
+
\Order(\sqrt{h})
\end{equation}
as $h\to0$ for any fixed~$v$.
\end{lemma}

\begin{proof}
We first write the function~$w$ as
\begin{equation}
w(y)
=
w(x)
+\Order(y^0)
+\Order(d(\bd y,\bd x)).
\end{equation}
The integral of the leading term $w(x)$ produces the desired integral by lemma~\ref{lma:d-int}, so it remains to estimate the integrals of the error terms.

For the $\Order(y^0)$ term the integral is
\begin{equation}
\int_{I_{v,h}}\Order(\gamma^0(t)^{+1/2})\der t
=
\abs{I_{v,h}}\Order(h^{1/2})
=
\Order(h).
\end{equation}
On the other hand, $d(\bd \gamma(t),\bd x)=\Order(\sqrt{h})$.
Therefore
\begin{equation}
\int_{I_{v,h}}\gamma^0(t)^{-1/2}\Order(d(\bd \gamma(t),\bd x))\der t
=
\Order(\sqrt{h})
\int_{I_{v,h}}\gamma^0(t)^{-1/2}\der t.
\end{equation}
By lemma~\ref{lma:d-int} the integral $\int_{I_{v,h}}\gamma^0(t)^{-1/2}\der t$ is $\Order(1)$, so the integral of the $\Order(d(\bd y,\bd x))$ error term is $\Order(\sqrt{h})$.
\end{proof}

\subsection{Proof of the boundary determination result}

We are now ready to prove theorem~\ref{thm:bdy-det}.

\begin{proof}[Proof of theorem~\ref{thm:bdy-det}]
It follows from \cite[Theorem 2.2]{MNP:bayes-xrt} and injectivity of the geodesic X-ray transform on simple manifolds that the normal operator is a bijection $I^*I\colon \delta^{-1/2}C^\infty(M)\to C^\infty(M)$, where $\delta\in C^\infty$ is positive in the interior and coincides with the distance function to the boundary in a neighborhood of the boundary.
However, our~$f$ is only assumed to be in~$L^1(M)$.

Suppose $f\in L^1(M)$ is such that $If\equiv1$.
Then $I^*If\equiv c$ for some constant~$c$ depending on dimension and normalization of measures, and as a constant function it is smooth. By \cite[Theorem 4.8]{MNP:bayes-xrt} our function~$f$ is of the form~\eqref{eq:f=dw} near the boundary for some $w\in C^\infty(M)$.

By lemma~\ref{lma:dw-int} the integral of~$f$ over the short geodesic~$\gamma_{v,h}$ is
\begin{equation}
\sqrt{\frac{2\pi^2}{\sff(v,v)}}w(x)
+
\Order(\sqrt{h}).
\end{equation}
This equals one for all $h$, so letting $h\to0$ gives
\begin{equation}
\sqrt{\frac{2\pi^2}{\sff(v,v)}}w(x)
=
1,
\end{equation}
and so
\begin{equation}
w(x)
=
\sqrt{\frac{\sff(v,v)}{2\pi^2}}.
\end{equation}
This has to hold for all $v\in S_x(\partial M)$.
Therefore the second fundamental form is independent of direction, meaning that~$\partial M$ is umbilical at~$x$.
The second fundamental form can thus be regarded as a scalar, and with that interpretation we obtain the claimed formula.
\end{proof}

\section{Euclidean characterization}
\label{sec:Rn}

\subsection{Projection slice theorem and integral moments}

We can parametrize lines in the plane by $(r,\theta)\in\R\times\R$ as $L_{r,\theta}=\{x\in\R^2;\ip{x}{v_\theta}=r\}$, where $v_\theta=(\cos\theta,\sin\theta)$.
The projection slice theorem states that if $f\in L^1_c(\R^2)$ and $h\in L^\infty_\loc(\R)$, then
\begin{equation}
\label{eq:pst1}
\int_{\R^2}f(x)h(v_\theta\cdot x)\der x
=
\int_{\R}If(r,\theta)h(r)\der r.
\end{equation}
This is a version of Fubini's theorem written in terms of the planar Radon transform; see~\cite[Theorem 2.2]{Q:intro}.

We consider our X-ray transform problem in a smooth, strictly convex, and bounded planar domain $\Omega\subset\R^2$.
We consider $a\colon\R/2\pi\Z\to\R$ so that
\begin{equation}
a(\theta)
=
\inf_{x\in\Omega}\ip{x}{v_\theta}
\end{equation}
and $b\colon\R/2\pi\Z\to\R$ similarly with infimum replaced by supremum.
The domain~$\Omega$ lies between the lines~$L_{a(\theta),\theta}$ and~$L_{b(\theta),\theta}$ and is tangent to them for any~$\theta$.
Since~$\Omega$ is bounded and smooth, these functions are well-defined and smooth.

\begin{lemma}
\label{lma:pst}
Let $\Omega\subset\R^2$ be a smooth, strictly convex, and bounded domain.
Suppose $f\in L^1(\Omega)$ satisfies $If\equiv 1$.
Then
\begin{equation}
\label{eq:pst2}
\int_{\Omega}f(x)h(v_\theta\cdot x)\der x
=
\int_{a(\theta)}^{b(\theta)}h(r)\der r
\end{equation}
for all $h\in L^\infty_\loc(\R)$.
\end{lemma}

\begin{proof}
The assumption $If\equiv1$ can be rephrased as
\begin{equation}
If(r,\theta)
=
\begin{cases}
1, & a(\theta)<r<b(\theta)\\
0, & \text{otherwise}.
\end{cases}
\end{equation}
The claim follows from equation~\eqref{eq:pst1} when $f$ is extended by zero to the whole plane.
\end{proof}

\begin{lemma}
\label{lma:R2}
Let $\Omega\subset\R^2$ be a smooth, strictly convex, and bounded domain.
Suppose $f\in L^1(\Omega)$ satisfies $If\equiv 1$.
Then $\Omega$ is a disc.
\end{lemma}

\begin{proof}
Applying lemma~\ref{lma:pst} with the constant test function $h(t)=1$ yields
\begin{equation}
\int_{\R^2}f(x)\der x
=
b(\theta)-a(\theta).
\end{equation}
Since this holds for all~$\theta$, we conclude that there is a constant $w>0$ so that $b(\theta)=a(\theta)+w$ for all~$\theta$.
That is, the domain has constant width~$w$.
The disc is, however, not the only planar domain of constant width.

Applying lemma~\ref{lma:pst} with the test function $h(t)=t$ leads to
\begin{equation}
\label{eq:vv2}
v_\theta\cdot\int_{\R^2}f(x)x\der x
=
\frac12w^2+wa(\theta).
\end{equation}
Since $\partial_\theta^2v_\theta=-v_\theta$, differentiating~\eqref{eq:vv2} twice with respect to~$\theta$ gives the ODE
\begin{equation}
wa''(\theta)+wa(\theta)
=
-\frac12w^2.
\end{equation}
The general solution of this this ODE is
\begin{equation}
a(\theta)=-\frac12w+\ip{z}{v_\theta},
\end{equation}
where $z\in\R^2$.
Shifting the domain by $-z$ leads to $a(\theta)=-\frac12w$ and consequently $b(\theta)=\frac12w$.
Therefore~$\Omega$ is a disc of radius~$\frac12w$.
\end{proof}

The very last conclusion of the proof is intuitively obvious, but we sketch a proof for the sake of completeness.
Suppose~$a$ and~$b$ are the constant functions~$\mp\frac12w$.
It follows from the definition of~$a$ and~$b$ that~$\Omega$ is contained in the disc.
If~$\Omega$ was strictly smaller, there would be a small dent and~$b$ would be smaller than~$\frac12w$ in that direction.
Therefore~$\Omega$ is a disc.

\subsection{Proof of the characterization of Euclidean domains}

We are now ready to prove theorem~\ref{thm:Rn}.

\begin{proof}[Proof of theorem~\ref{thm:Rn}]
Suppose there is $f\in L^1(\Omega)$ so that $If\equiv1$.
If $n=2$, lemma~\ref{lma:R2} implies that~$\Omega$ is a disc.
In any dimension, theorem~\ref{thm:bdy-det} implies that~$\partial\Omega$ is umbilical.
If $n\geq3$, this implies that~$\Omega$ is a ball, as the only umbilical hypersurfaces in $\R^n$, $n\geq3$, are spheres and hyperplanes; see e.g.~\cite[Exercise 8.6c]{dC:RG}.

It remains to show that if~$\Omega$ is a ball, then a function $f\in L^1(\Omega)$ with $If\equiv1$ indeed exists.
Uniqueness of such~$f$ follows from injectivity of the X-ray transform.
The property is invariant under scaling and translation, and in the unit ball the function is given by
\begin{equation}
\label{eq:radial-f}
f(x)
=
\frac{1}{\pi\sqrt{1-\abs{x}^2}},
\end{equation}
as one can easily verify by direct calculation.
\end{proof}

\subsection{An alternative proof with partial data}

The X-ray transform is critically determined in two dimensions, so it is natural to use full data.
However, in higher dimensions the transform is overdetermined.
Indeed, the proof of theorem~\ref{thm:Rn} given above relies almost entirely on short line segments and a boundary determination result.
We used long lines to argue that~$f$ is of the form~\eqref{eq:f=dw} by way of the mapping properties if~$I^*I$.
The use of these lines can be circumnavigated, which leads to a stronger version of the theorem.
The proof is based on two geometrical lemmas.

Heuristically, the proof can be understood as follows:
If~$f$ integrates to~$1$ over all geodesics, this also holds in all two-dimensional planes intersected with the original domain.
By lemma~\ref{lma:R2} (the planar case of theorem~\ref{thm:Rn}) these intersections have to be discs.
When the plane is almost tangent to the boundary, the intersection is approximately an ellipse whose shape is determined by the second fundamental form.
Because the intersections are discs, the ellipse has to be a ball, and so the boundary is umbilical.

\begin{lemma}
\label{lma:line-length}
Let $\Omega\subset\R^n$ be a strictly convex smooth domain and $x\in\partial\Omega$.
Let~$\nu$ the inward unit normal and~$v$ any unit tangent vector to~$\partial\Omega$ at~$x$.
For $h>0$ we denote by~$\ell_v^h$ the length of the line segment $\Omega\cap(h\nu+v\R)$.

We have
\begin{equation}
\lim_{h\to0}\frac{(\ell_v^h)^2}{h}
=
\frac{8}{\sff(v,v)}.
\end{equation}
\end{lemma}

\begin{proof}
It follows from~\cite[Lemma 12]{I:bdy-det} that
\begin{equation}
\abs{
\frac{\ell_v^h}{\sqrt{8h/\sff(v,v)}}
-1
}
=
\Order(\sqrt{h}).
\end{equation}
The reference time~$t_0$ of the cited lemma can be taken to be at the tip of the short geodesic rather than the base.

Therefore, as $h\to0$, we have
\begin{equation}
\frac{\ell_v^h}{\sqrt{h}}\times\sqrt{\frac{\sff(v,v)}{8}}
\to
1.
\end{equation}
This proves the claimed limit.
\end{proof}

\begin{lemma}
\label{lma:umbilical-slice}
Let $\Omega\subset\R^n$, $n\geq3$, be a strictly convex smooth domain and $x\in\partial\Omega$.
Let~$\nu$ the inward unit normal to~$\partial M$ at~$x$. 
If for some $h_0>0$ the intersection $(P+h\nu)\cap\Omega$ is a disc for almost all $h\in(0,h_0)$ and almost all two-planes~$P$ tangent to~$\partial\Omega$ at~$x$, then~$\partial\Omega$ is umbilical at~$x$.
\end{lemma}

\begin{proof}
Take any such two-plane~$P$ that $(P+h\nu)\cap\Omega$ is a disc for almost all $h\in(0,h_0)$.
Take any two unit vectors $v_1,v_2$ in~$P$.
The line segments $(v_i\R+h\nu)\cap\Omega$ have therefore equal length.
Letting $h\to0$ also shows that lengths scaled by~$\sqrt{h}$ must be equal.
By lemma~\ref{lma:line-length} this means that $\sff(v_1,v_1)=\sff(v_2,v_2)$.

This holds for almost all two-planes $P$ by assumption, so $\sff(v,v)=\sff(u,u)$ for almost all pairs of unit vectors $u,v\in T_x(\partial\Omega)$.
By continuity it holds for all pairs and so~$\partial\Omega$ is umbilical at~$x$.
\end{proof}

A localized version of theorem~\ref{thm:Rn} can be formulated as follows.

\begin{theorem}
\label{thm:Rn-local}
Let $\Omega\subset\R^n$, $n\geq3$, be a bounded, smooth, and strictly convex domain.
Fix any $\eps>0$.
If $f\in L^1(\Omega)$ integrates to zero over almost all geodesics of length less than~$\eps$, then~$\Omega$ is a ball.
\end{theorem}

\begin{proof}
Let us use the notation of the preceding lemmas.

By assumption, for almost every direction on the unit sphere the function integrates to one over almost every short geodesic normal to the chosen direction.
By strict convexity the map from a boundary point to the corresponding normal direction is a diffeomorphism $\partial\Omega\to S^{n-1}$.
If a hyperplane is close enough to being tangent to~$\partial\Omega$, all maximal geodesics contained in it have length less than~$\eps$.
By compactness the required closeness can be chosen to be uniform.
Therefore there is some $h_0>0$ so that for almost every $x\in\partial\Omega$ the following holds for almost every $h\in(0,h_0)$:
The function~$f$ is~$L^1$ in the intersection $(P+h\nu)\cap\Omega$ and its integral over almost every line is~$1$.

By lemma~\ref{lma:R2} this implies that every such $(P+h\nu)\cap\Omega$ is a disc.
We are thus in the setting of lemma~\ref{lma:umbilical-slice}, and we conclude that~$\partial\Omega$ is umbilical almost everywhere.
By continuity it is umbilical everywhere.
Therefore~$\Omega$ is a ball.
\end{proof}

\section{Radial manifolds}
\label{sec:radial}

\begin{proof}[Proof of proposition~\ref{prop:radial}]
By symmetry, it suffices to consider the two-dimensional case.
Namely, any geodesic is contained in a two-dimensional plane through the origin.
Given any initial conditions~$\gamma(0)$ and~$\dot\gamma(0)$, the plane is the one spanned by these two.
(The geodesic is radial if the two are linearly dependent.)
To see this, observe that reflection across this plane is an isometry and maps geodesics to geodesics but the initial conditions are invariant under it, whence the whole geodesic must be invariant under this reflection.

The closed Euclidean unit disc $D\subset\R^2$ is equipped with the metric $c^{-2}(r)g_\Euc$, where~$g_\Euc$ is the Euclidean metric.
We denote $\rho(r)=r/c(r)$, so that the Herglotz condition becomes $\rho'(r)>0$.
We seek a radial function~$f(r)$ so that $If\equiv1$.

As shown in \cite[Section 5.2]{dHI:abel} (and in a different notation earlier in~\cite{S:revolution}) that the integral of~$f$ over a geodesic whose radius (smallest value of Euclidean distance to the origin) is at $s\in(0,1)$ is
\begin{equation}
\A f(s)
=
2\int_s^1 f(r)
\left[
1-\left(\frac{\rho(s)}{\rho(r)}\right)^2
\right]^{-1/2}
\frac{\der r}{c(r)}.
\end{equation}
The Abel transform~$\A$ related to~$\rho$ has an explicit inversion formula~\cite{dHI:abel}, valid for any $f\in L^1(0,1)$:
\begin{equation}
\label{eq:abel-inv}
f(r)
=
-\frac{c(r)}{\pi}
\left.
\Der{x}
\int_x^1
\frac{\rho'(x)}{\rho(x)}
\left[
\left(\frac{\rho(z)}{\rho(x)}\right)^2-1
\right]^{-1/2}
\A f(z)
\der z
\right|_{x=r}.
\end{equation}
The desired property is that $\A f\equiv1$, and it is easy to verify that the function arising from this formula is indeed~$L^1$.

Using $Af\equiv1$ and changing the variable of integration from~$r$ to $\tilde r=\rho(r)$ shows that the integral of~\eqref{eq:abel-inv} is
\begin{equation}
\int_x^1
\frac{\rho'(z)}{\rho(z)}
\left[
\left(\frac{\rho(z)}{\rho(x)}\right)^2-1
\right]^{-1/2}
\der z
=
\frac\pi2-\arcsin(\rho(x)/\rho(1)).
\end{equation}
Differentiating with respect to~$x$ gives then
\begin{equation}
f(r)
=
\frac{c(r)}{\pi}
\rho'(r)
\frac1{\sqrt{\rho(1)^2-\rho(r)^2}}.
\end{equation}
Using the definition of~$\rho$ and simplifying gives the claimed formula.
\end{proof}

When $c\equiv1$, one can easily verify that the obtained radial function is exactly that of~\eqref{eq:radial-f}.

\end{document}